\documentclass[a4paper,10pt]{amsart}
\usepackage{amssymb,amscd,amsmath}

\theoremstyle{plain}
\newtheorem{thm}{Theorem}[section]
\newtheorem{prop}[thm]{Proposition}
\newtheorem{lem}[thm]{Lemma}

\theoremstyle{definition}

\newtheorem{prob}{Question}[section]

\theoremstyle{remark}
\newtheorem{rem}{Remark}[section]
\newtheorem{ex}[rem]{Example}

\newcommand{\forme}[1]{}

\newcommand{\Ot}{{\mathbf O}_{\mathrm \theta}}
\newcommand{\Or}{{\mathbf O}^{\mathrm \theta}}

\title{On commutative $p$-schemes of order $p^4$}

\author[K.~Kim]{Kijung Kim}
\address{Department of Mathematics, Pusan National University, Busan 46241, Republic of Korea}
\email{knukkj@pusan.ac.kr}

\date{\today}

\begin{document}
\maketitle

\begin{abstract}
In this article, we consider the existence and schurity problem on commutative $p$-schemes of order $p^4$.
Using the thin radical and thin residue, we give sufficient conditions for such $p$-schemes to be schurian.
We also give questions related to our results.
\end{abstract}

\footnotetext{This research was supported by Basic Science Research Program through the National Research Foundation of Korea(NRF) funded by the Ministry of Education (2013R1A1A2005349).}

\section{Introduction}\label{sec:intro}

An association scheme is a combinatorial object which is defined by
some algebraic properties derived from a transitive permutation group (see Section \ref{sec:pre} for definitions).
So, we can regard association schemes as a generalization of groups.
In this sense, $p$-schemes correspond to $p$-groups (see Section \ref{sec:pre} for definitions).
It is known that all $p$-schemes of order $p$ are unique up to isomorphism.
Unlike $p$-groups of order $p^2$, the number of isomorphism classes of $p$-schemes of order $p^2$ is $3$.
The classification of $p$-schemes of order $p^3$ is far from being complete.

As we mentioned the above, every transitive permutation group $G$ on a finite set $X$ induces an association scheme $\mathcal{R}_G$,
where $\mathcal{R}_G$ is the set of orbits by the componentwise action of $G$ on $X \times X$.
We say that an association scheme $S$ is \textit{schurian} if $S = \mathcal{R}_G$ for a transitive permutation group $G$ on $X$.
Characterizing schurian association schemes is one of the major topics in the theory of association schemes.
In the case of $p$-schemes, one can check that every $p$-scheme of order at most $p^2$ is schurian.
In \cite{chk,kim20130401}, some non-schurian $p$-schemes of order $p^3$ are given.

For a given association scheme, we can define its thin radical $\Ot(S)$ and thin residue $\Or(S)$, respectively (see Section \ref{sec:pre} for definitions).
We denote the orders of $\Ot(S)$ and $\Or(S)$ by $n_{\Ot(S)}$ and $n_{\Or(S)}$, respectively.
According to \cite{hanakimi}, all of non-schurian commutative $2$-schemes of order $16$ are \texttt{as-16.no.55, 59, 79, 85, 89, 90, 94, 95}.
Note that they satisfy $n_{\Ot(S)}=2$ and $n_{\Or(S)}=8$.
This article is concerned with commutative $p$-schemes of order $p^4$, where $p$ is an odd prime.
We can construct a non-schurian commutative $p$-scheme of order $p^4$ with $n_{\Ot(S)}=p$ (see Example \ref{example:2}).
Our motivation is to find non-schurian $p$-schemes except for the case $n_{\Ot(S)}=p$.
Such attempt leads to classifying by schurian subclasses.
Our main result is the following.

\begin{thm}\label{thm:main-intro}
Let $S$ be a commutative $p$-scheme of order $p^4$.
Assume that one of the following conditions holds.
\begin{enumerate}
\item $\Or(S) = \Ot(S)$, $n_{\Ot(S)} = p^2$ and $S/\!/\Or(S) \cong C_{p^2}$.
\item $n_{\Ot(S)} = n_{\Or(S)} =p^2$, $\Or(S) \neq \Ot(S)$ and $n_s = p^2$ for each $s \in S \setminus \Ot(S)\Or(S)$.
\item $n_{\Ot(S)}=p^3$.
\end{enumerate}
Then $S$ is schurian.
\end{thm}
Note that we prove the above result in Theorems \ref{thm:cp2}, \ref{thm:2-2-s} and \ref{thm:p^n-1}.


This article is organized as follows.
In Section \ref{sec:pre}, we prepare some terminology and notation.
In Section \ref{sec:main1}, we give our main results.
In Section \ref{sec:main2}, we prove Theorem \ref{thm:2-2-s}.

\section{Preliminaries}\label{sec:pre}

Let $X$ be a nonempty set, and let $S$ be a partition of $X \times X$.
The set $S$ is called an \textit{association scheme} (or shortly  a \textit{scheme}) on $X$ if it satisfies the following conditions:
\begin{enumerate}
\item $1_X := \{ (x, x) \mid  x \in X \} \in S$;
\item For each $s \in S$, $s^\ast := \{(x, y) \mid (y, x) \in s \} \in S$;
\item For all $ s, t, u \in S$ and $x, y \in X$, $a_{stu} := |xs \cap yt^\ast|$ is constant whenever $(x, y) \in u$,
where $xs:=\{ y \in X \mid (x,y) \in s \}$.
\end{enumerate}

For each $s$ in $S$, we set $n_s:= a_{ss^*1_X}$ and call this (positive) integer the \textit{valency} of $s$.
The unique relation containing a pair $(x, y) \in X \times X$ is denoted by $r(x, y)$.
For a subset $U$ of $S$, put $n_{_U} := \sum_{u \in U} n_u$.
We call $n_{_U}$ the \textit{order} of $U$.
The scheme $S$ is called \textit{p-valenced} if the valency of every element is a power of $p$, where $p$ is a prime.
In particular, a $p$-valenced scheme $S$ is called a \textit{p-scheme} if $|X|$ is also a power of $p$.

Let $P$ and $Q$ be nonempty subsets of $S$. We define $PQ$ to be the set of all elements $s \in S$ such that there exist elements
$p \in P$ and $q \in Q$ with $a_{pqs}\neq 0$. The set $PQ$ is called the \textit{complex product} of $P$ and $Q$.
If one of factors in a complex product consists of a single element $s$, then one usually writes  $s$ for $\{ s \}$.

A nonempty subset $T$ of $S$ is called \textit{closed} if $TT \subseteq T$.
Note that a subset $T$ of $S$ is closed if and only if $\bigcup_{t \in T}t$ is an equivalence relation on $X$.
A closed subset $T$ is called \textit{thin} if all elements of $T$ have valency 1.
The set $\{ s \mid n_s=1 \}$ is called the \textit{thin radical} of $S$ and denoted by $\Ot(S)$.
Note that $T$ is thin if and only if $T$ is a group under the relational product.

Let $Y$ be a subset of $X$. For each $s \in S$, we define $s_Y := s \cap (Y \times Y)$.
For each closed subset $T$ of $S$, we set $T_Y := \{ t_Y \mid t \in T \}$.
Let $x$ be an element in $X$, and $T$ be a closed subset of $S$.
Then $T_{xT}$ is an association scheme on $xT:= \bigcup_{t \in T} xt$, which is called \textit{subscheme} of $S$ defined by $xT$ (see \cite[Theorem 2.1.8]{zies2}).

A closed subset $T$ of $S$ is called \textit{strongly normal} in $S$, denoted by $T \lhd^\sharp S$, if $s^* T s \subseteq T$ for every $s \in S$.
We put $\Or(S) := \bigcap_{T \lhd^\sharp S} T $ and call it the \textit{thin residue} of $S$.
Note that $\Or(S) = \langle \bigcup_{s \in S} s^*s \rangle$ (see \cite[Theorem 2.3.1]{zies}).

For each closed subset $T$ of $S$, we define $X / T:= \{ x T \mid x \in X \}$ and $S/\!/T:=\{ s^T \mid s \in S \}$,
where $s^T:= \{ (xT, yT) \mid y \in xTsT \}$. Then $S/\!/T$ is an association scheme on $X/T$,
which is called the \textit{quotient} (or \textit{factor}) scheme of $S$ over $T$ (see \cite[Theorem 4.1.3]{zies2}).
Note that $T \lhd^\sharp S$ if and only if $S/\!/T$ is a group (see \cite[Theorem 2.2.3]{zies}).

Let $S_1$ be an association scheme on $X_1$. A bijective map $\phi$ from $X \cup S$ to $X_1 \cup S_1$ is called an \textit{isomorphism} if it satisfies the following conditions:
\begin{enumerate}
\item $X^\phi \subseteq X_1$ and $S^\phi \subseteq S_1$;
\item For all $x, y \in X$ and $s \in S$ with $(x,y) \in s$, $(x^\phi, y^\phi) \in s^\phi$.
\end{enumerate}
An isomorphism $\phi$ from $X \cup S$ to $X \cup S$ is called an \textit{automorphism} of $S$
if $s^{\phi}=s$ for all $s \in S$.
We denote by $\mathrm{Aut}(S)$ the automorphism group of $S$.
On the other hand, we say that $S$ and  $S_1$ are \textit{algebraically isomorphic} or \textit{have the same intersection numbers} if there exists a bijection $\iota$ from $S$ to $S_1$ such that $a_{rst} = a_{r^\iota s^\iota t^\iota}$ for all $r, s, t \in S$.

Let $F$ and $H$ be association schemes on $W$ and $Y$, respectively.
For each $f \in F$ we define
\[ \overline{f}:= \{ ((w_1,y),(w_2,y)) \mid y \in Y, (w_1,w_2) \in f  \} .\]
For each $h \in H \setminus \{1_Y\}$ we define
\[ \overline{h}:=\{ ((w_1,y_1),(w_2,y_2)) \mid w_1, w_2 \in W, (y_1,y_2) \in h \}. \]
Denote $F \wr H := \{ \overline{f} \mid f \in F \} \cup \{\overline{h} \mid h \in H \setminus \{1_Y\} \}$.
Then $F \wr H$ is an association scheme on $W \times Y$, which is called the \textit{wreath product} of $F$ and $H$.
We note that if $S$ is the wreath product of $T_{xT}$ and $S/\!/T$ for some closed subset $T$ of $S$,
then we simply denote $S$ by $T \wr (S/\!/T)$ instead of $T_{xT} \wr (S/\!/T)$.

For each $s \in S$, we denote by $\sigma_s$ the \textit{adjacency matrix} of $s$.
Namely $\sigma_s$ is a matrix whose rows and columns are indexed by the elements of $X$
and $(\sigma_s)_{x y} = 1$ if $(x, y) \in s$ and $(\sigma_s)_{x y} = 0$ otherwise.

We define the \textit{left stabilizer} and \textit{right stabilizer} of $s \in S$ by
\[L(s)=\{ t \in S \mid ts=s \} ~\text{and} ~~ R(s)=\{ t \in S \mid st=s \}.\]

A map $\phi$ from a subset $Y$ of $X$ to $X$ is called \textit{faithful} if
$r(x, y) = r(x^\phi, y^\phi)$ for $x, y \in Y$ (see \cite{zies2}).

\begin{rem}\label{rem:transitive}
For any $x, y \in X$ there exists a faithful map $\phi$ from $X$ to $X$ such that $x^\phi = y$
if and only if $\mathrm{Aut}(S)$ is transitive on $X$.
\end{rem}

For $C, D \subseteq X$ and $\phi \in Sym(X)$, we say that $C$ and $D$ are
\textit{compatible with respect to $\phi$} if
\[ r(x,y) = r(x^\phi, y^\phi) ~\text{for each}~ (x,y) \in C \times D.\]
We shall write $C \sim_\phi D$ if $C$ and $D$ are compatible with respect to $\phi$, otherwise $C \nsim_\phi D$.

The following lemma is a collection of basic facts.
\begin{lem}[See \cite{afm,zies}]\label{lem:basic}
Let $S$ be an association scheme on $X$. For $u, v, w \in S$, we have the following:
\begin{enumerate}
\item $a_{u^* v 1_X} = \delta_{u,v} n_u$;
\item $a_{uvw} = a_{v^* u^* w^*}$;
\item $a_{uvw^*} n_w= a_{vwu^*} n_u = a_{wuv^*} n_v$;
\item $n_u n_v = \sum_{s \in S} a_{uvs} n_s$;
\item $\mathrm{lcm}(n_u, n_v) | a_{uvw} n_w$;
\item $\mathrm{gcd}(n_u, n_v) \geq |uv|$.
\end{enumerate}
\end{lem}

\begin{thm}[See Theorem B of \cite{zi}]\label{cyclic-phz}
Assume that $\Or(S) \subseteq \Ot(S)$ and that $\{ s^\ast s \mid s \in S\}$
is linearly ordered with respect to set-theoretic inclusion. Then $S$ is schurian.
\end{thm}

The following is well known.

\begin{thm}\label{wreath-product}
Let $S_1$ and $S_2$ be association schemes.
Then $S_1$ and $S_2$ are schurian if and only if $S_1 \wr S_2$ is schurian.
\end{thm}

\begin{lem}[See Lemma 3.3 of \cite{chk}]\label{chk:Lemma 3.3}
Let $S$ be a $p$-scheme of order $p^3$ such that $\Or(S) \cong C_p \times C_p$.
Then $S$ is commutative if and only if $\Or(S)s = s$ for each $s \in S \setminus \Or(S)$.
\end{lem}

\section{Main results}\label{sec:main1}
Throughout this section, we assume that $S$ is a commutative $p$-scheme of order $p^4$, where $p$ is an odd prime.
We divide our consideration into three subsections depending on $n_{\Ot(S)}$.

\subsection{The case of $n_{\Ot(S)}=p$}
\subsubsection{$n_{\Ot(S)}=p$ and $n_{\Or(S)} =p$}
\hfill \break
Since $\Or(S)$ is a cyclic group, it follows from Theorem \ref{cyclic-phz} that $S$ is schurian.

\vskip10pt
\subsubsection{$n_{\Ot(S)}=p$ and $n_{\Or(S)} =p^3$}
\hfill \break
In this case, we give a non-schurian example.

\begin{ex}\label{example:2}
Let $T$ be a non-schurian commutative $p$-scheme of order $p^3$ such that $n_{\Ot(T)}=p$ and $n_{\Or(T)} =p^2$ (see \cite[Theorem 4.3]{chk}).
Then $T \wr C_p$ is a non-schurian commutative $p$-scheme of order $p^4$ such that $n_{\Ot(T\wr C_p)}=p$ and $n_{\Or(T\wr C_p)} =p^3$.
\end{ex}



\vskip20pt
\subsection{The case of $n_{\Ot(S)}=p^2$}
\subsubsection{$n_{\Ot(S)}=p^2$ and $n_{\Or(S)} =p$}
\hfill \break
Since $C_p \cong \Or(S) \subseteq \Ot(S)$, it follows from Theorem \ref{cyclic-phz} that $S$ is schurian.

\vskip10pt
\subsubsection{$n_{\Ot(S)} =n_{\Or(S)} =p^2$ and $\Or(S) = \Ot(S)$}
\hfill \break
In this case, $S$ is isomorphic to a fission of the wreath product of two thin schemes (see \cite[Theorem 1]{rr}).

\begin{thm}\label{thm:cp2}
Let $S$ be a commutative $p$-scheme on $X$ such that $\Or(S) = \Ot(S)$, $n_{\Ot(S)} = p^2$ and $S/\!/\Or(S) \cong C_{p^2}$.
Then $S$ is schurian.
\end{thm}
\begin{proof}
If $\Or(S) \cong C_{p^2}$, then it follows from Theorem \ref{cyclic-phz} that $S$ is schurian.

Suppose $\Or(S) \cong C_p \times C_p$.
Since $S/\!/\Or(S) \cong C_{p^2}$, there exists a unique closed subset $T$ such that $\Or(S) \subseteq T$ and $n_T = p^3$.
By the definition of thin residue, we have $\Or(T) \subseteq \Or(S)$.
We divide our consideration into two cases : $\Or(T) = \Or(S)$ and $\Or(T) \neq \Or(S)$.

\vskip5pt
\textbf{(Case 1)} $\Or(T) = \Or(S)$.

For $x \in X$, $T_{xT}$ is a commutative $p$-scheme on $xT$ of order $p^3$ such that $\Or(T_{xT}) = \Ot(T_{xT}) \cong C_p \times C_p$.
By Lemma \ref{chk:Lemma 3.3}, we have $n_t =p^2$ for each $t \in T \setminus \Or(S)$.
Note that $T_{xT}$ is schurian.

\vskip5pt
In the rest of (Case 1), we shall show $n_s = p^2$ for each $s \in S \setminus T$.

Suppose that there exists $s \in S \setminus T$ such that $n_s =p$.
If $|R(s)|=1$, then $n_{\Ot(S)s} = p^3$ and
$\sigma_{s^\ast} \sigma_s = p\sigma_{1_X} + \sum_{t \in T \setminus \Ot(S)} a_{s^\ast s t} \sigma_t$.
By Lemma \ref{lem:basic}(iv), there exist $t \in T \setminus \Ot(S)$ such that $a_{s^\ast s t} \neq 0$.
This contradicts $\Or(S) = \Ot(S)$. Thus, we have $|R(s)|=p$.

\vskip5pt
\textbf{Claim} : $\sigma_s \sigma_s = p \sigma_{s'}$ for some $s' \in S$ with $n_{s'}=p$.

Let $s' \in ss$, $(\alpha, \beta) \in s'$ and $R(s)= \langle l_1 \rangle$.
Then there exists $\gamma_1 \in \alpha s \cap \beta s^\ast$.
Using $sl_1 =s$, for $(\alpha, \gamma_1) \in s$ we have $\gamma_2 \in \alpha s \setminus \{\gamma_1\}$ such that
$\gamma_2 \in \alpha s \cap \gamma_1 l_1^\ast$.
It follows from $l_1 s =s$ that $\beta \in \gamma_1 s = \gamma_2 l_1 s = \gamma_2 s$.
Thus, we have $\gamma_2 \in \alpha s \cap \beta s^\ast$.

Since $s l_1^i = s$ for $l_1^i$ $(2 \leq i \leq p-1)$, by applying the same argument for $l_1^i$ we obtain $|\alpha s \cap \beta s^\ast|=p$.
This completes the proof of Claim.

\vskip5pt
By Lemma \ref{lem:basic}(iv) and Claim, we have $ss=s'$.
Since $S/\!/\Ot(S) \cong C_{p^2}$ and $s \not\in \Ot(S)$, we have $s \Ot(S) \cap s' \Ot(S) = \emptyset$.
Note that $R(s) = R(s')$.

\vskip5pt
By the argument used in the proof of Claim, we can show $\sigma_{s'} \sigma_s = p \sigma_{s''}$ for some $s'' \in S$ with $n_{s''}=p$.
Note that $R(s) = R(s'')$ and $s''=sss$.
By repeating this process, we have $s^p \in T \setminus \Or(S)$ with $n_{s^p} = p$, since $S/\!/\Or(S) \cong C_{p^2}$.
But, this contradicts the fact that $n_t =p^2$ for each $t \in T \setminus \Or(S)$.

Therefore, we have $n_s = p^2$ for each $s \in S \setminus \Or(S)$, i.e., $S \cong \Or(S)\wr C_{p^2}$.
By Theorem \ref{wreath-product}, $S$ is schurian.

\vskip5pt
\textbf{(Case 2)} $\Or(T) \neq \Or(S)$.

Since $T$ is not thin and $\Or(T) \subset \Or(S)$, we have $n_{\Or(T)} = p$.
This implies $|\{ t^\ast t \mid t \in T \setminus \Ot(S) \}|=1$.

Suppose that there exists $s \in S \setminus T$ such that $n_s =p$.
Then it is easy to see $|R(s)|=p$.
By the same argument of (Case 1), we have $s^p \in T \setminus \Or(S)$.
This implies $|\{ R(s) \mid s \in S \setminus \Ot(S) \}|=1$.
Thus, we have $n_{\Or(S)} =p$, a contradiction.

Therefore, we have $n_s = p^2$ for each $s \in S \setminus T$.
Since $s^\ast s = \Or(S)$ and $t^\ast t = \Or(T)$ for all $s \in S \setminus T$ and $t \in T \setminus \Or(S)$,
$\{s^\ast s \mid s \in S\}$ is linearly ordered.
It follows from Theorem \ref{cyclic-phz} that $S$ is schurian.
\end{proof}

The following example is a schurian commutative $p$-scheme such that $\Or(S) = \Ot(S) \cong C_p \times C_p$ and $S/\!/\Or(S) \cong C_p \times C_p$.



\begin{ex}(See \cite[Subsection 4.1]{hk})\label{example:4}
There exists a schurian commutative $p$-scheme such that $\Or(S) = \Ot(S) \cong C_p \times C_p$, $S/\!/\Or(S) \cong C_p \times C_p$ and $n_s =p$ for each $s \in S \setminus \Ot(S)$.
\end{ex}

\begin{prob}\label{que:1}
Is there a non-schurian $p$-scheme algebraically isomorphic to Example \ref{example:4} ?
\end{prob}

\vskip10pt
\subsubsection{$n_{\Ot(S)} =n_{\Or(S)} =p^2$ and $\Or(S) \neq \Ot(S)$}
\hfill \break
Since $\Or(S)$ has a nontrivial thin closed subset, we have $|\Or(S) \cap \Ot(S)|=p$.
By \cite[Lemma 2.1.1]{zies2}, $\Or(S)\Ot(S)$ is a closed subset.
Also, we have $n_{\Or(S)\Ot(S)} = p^3$.



\begin{thm}\label{thm:2-2-s}
Let $S$ be a commutative $p$-scheme on $X$ such that $n_S = p^4$, $n_{\Ot(S)} = n_{\Or(S)} =p^2$ and $\Or(S) \neq \Ot(S)$.
Assume $n_s = p^2$ for each $s \in S \setminus \Ot(S)\Or(S)$.
Then $S$ is schurian.
\end{thm}

In Section \ref{sec:main2}, we give our proof of Theorem \ref{thm:2-2-s}.

\begin{prob}\label{que:2-1}
In Theorem \ref{thm:2-2-s}, is it possible to exist $s \in S \setminus \Ot(S)\Or(S)$ with $n_s =p$ ?
\end{prob}










\vskip10pt
\subsubsection{$n_{\Ot(S)}=p^2$ and $n_{\Or(S)} = p^3$}

\begin{lem}\label{lem:thin-r-thin-r}
Let $S$ be a commutative $p$-scheme of order $p^4$ such that $n_{\Ot(S)}=p^2$ and $n_{\Or(S)} = p^3$.
Then $\Ot(S) \subseteq \Or(S)$.
\end{lem}
\begin{proof}
Suppose $\Ot(S) \not\subseteq \Or(S)$.
Since $\Or(S)$ contains a nontrivial thin closed subset, we have $|\Or(S) \cap \Ot(S)| = p$ and $\Or(S)\Ot(S) = S$.
The last equation implies $\Or(S) = \langle s^\ast s \mid s \in \Or(S)\Ot(S)\rangle = \langle s^\ast s \mid s \in \Or(S)\rangle = \Or(\Or(S))$.
This contradicts $\Or(\Or(S)) \subset \Or(S)$ (see \cite[Theorem 2.4.6]{zies}).
\end{proof}

\begin{prop}\label{3-1}
Let $S$ be a commutative $p$-scheme of order $p^4$ such that $n_{\Ot(S)}=p^2$, $n_{\Or(S)} = p^3$
and $n_t =p^2$ for each $t \in \Or(S) \setminus \Ot(S)$.
Then we have $n_s \geq p^2$ for each $s \in S \setminus \Or(S)$.
\end{prop}
\begin{proof}
Suppose that there exists $s \in S \setminus \Or(S)$ with $n_s=p$.
Then $|R(s)|= 1$ or $p$.

When $|R(s)|= 1$, we have $n_{s\Ot(S)}=p^3$.
Also, we have $\sigma_{s^\ast} \sigma_s = p\sigma_{1_X} + \sum_{t \in \Or(S) \setminus \Ot(S)}  a_{s^\ast st}\sigma_t$.
Since every $t \in \Or(S) \setminus \Ot(S)$ has the valency $p^2$,
this is impossible.

When $|R(s)|= p$, we have $n_{s\Ot(S)} = p^2$.
Since $n_{s\Or(S)} = p^3$ and $\Ot(S) \subseteq \Or(S)$, we have $n_{s'} = p$ or $p^2$ for each $s' \in s\Or(S) \setminus \{ s \}$.

If $n_{s'} = p$, then it follows from the previous argument that $|R(s')|= p$.

If $n_{s'} = p^2$, then $|R(s')| \leq p^2$.
Since $n_{s'\Or(S)} = p^3$ and $n_{\Ot(S)}=p^2$, we have $|R(s')| \geq p$.
Suppose $|R(s')| = p$.
Then $s \Or(S) = s' \Or(S)$ but every element of $s' \Or(S)$ has the valency $p^2$, a contradiction.
So, we have $|R(s')| = p^2$.

Thus, we have $s\Or(S) = \bigcup_{s' \in s\Or(S)} s'\Ot(S)$ and $n_{s'\Ot(S)}=p^2$ for each $s' \in s\Or(S)$.
This implies that $S/\!/\Ot(S)$ must be a thin $p$-scheme of order $p^2$.
So, we have $\Or(S) \subseteq \Ot(S)$, a contradiction.

Therefore, we have $n_s \geq p^2$ for each $s \in S \setminus \Or(S)$.
\end{proof}

\begin{prob}\label{que:3}
Is there a commutative $p$-scheme $S$ of order $p^4$ such that $n_{\Ot(S)}=p^2$, $n_{\Or(S)} = p^3$
and $n_s =p^2$ for each $s \in S \setminus \Ot(S)$ ?
\end{prob}

\vskip20pt
\subsection{The case of $n_{\Ot(S)}=p^3$}
\hfill \break
In this subsection, we consider commutative $p$-schemes with $n_{\Ot(S)}=p^{n-1}$ and $n_S =p^n$ in general form.
Note that $\Or(S) \subseteq \Ot(S)$ since $S/\!/\Ot(S) \cong C_p$.


\begin{thm}\label{thm:p^n-1}
Let $S$ be a commutative $p$-scheme on $X$ of order $p^n$ such that $n_{\Ot(S)}=p^{n-1}$.
Then $S$ is schurian.
\end{thm}
\begin{proof}
We take $s \in S \setminus \Ot(S)$ such that $\mathrm{min}\{ n_t \mid t \in S \setminus \Ot(S) \} = n_s$.
Then $n_s = p^i$ for some $1 \leq i \leq n-1$.
Note that $|R(s)| \leq p^i$ and $R(s) = L(s)$.

First of all, we show $|R(s)|=p^i$.
Suppose $|R(s)| \leq p^{i-1}$.
By Lemma \ref{lem:basic}(v), for each $t \in \Ot(S)$ $st$ is a relation with valency $p^i$.
Since $|\Ot(S)/\!/R(s)|\geq p^{n-i}$, we have $n_{s\Ot(S)} \geq p^i \cdot p^{n-i} = p^n$, a contradiction.

\vskip5pt
\textbf{Claim} : $\sigma_s \sigma_s = n_s \sigma_{s'}$ for some $s' \in S$ with $n_{s'}=p^i$.

Let $(\alpha, \beta) \in s' \in ss$ and $l_1  \in R(s)$.
Then there exists $\gamma_1 \in \alpha s \cap \beta s^\ast$.
Since $sl_1 =s$, for $(\alpha, \gamma_1) \in s$, we have $\gamma_2 \in \alpha s \cap \gamma_1 l_1^\ast$ for some $\gamma_2 \in \alpha s \setminus \{\gamma_1 \}$.
It follows from $l_1 s =s$ that $\beta \in \gamma_1 s = \gamma_2 l_1 s = \gamma_2 s$.
Thus, we have $\gamma_2 \in \alpha s \cap \beta s^\ast$.

Applying the same argument for each $l \in R(s) \setminus \{l_1\}$, we obtain $|\alpha s \cap \beta s^\ast|=p^i$.
This completes the proof of Claim.

\vskip5pt
By Lemma \ref{lem:basic}(iv) and Claim, we have $ss=s'$.
Since $S/\!/\Ot(S) \cong C_p$ and $s \not\in \Ot(S)$, we have $s \Ot(S) \cap s' \Ot(S) = \emptyset$.
Note that $R(s) = R(s')$.

\vskip5pt
By the argument used in the proof of Claim, we can show $\sigma_{s'} \sigma_s = p^i \sigma_{s''}$ for some $s'' \in S$ with $n_{s''}=p^i$.
By repeating this process, we have $S = \bigcup_{0 \leq j \leq p-1} s^j \Ot(S)$.
This implies $\Or(S) = s^\ast s$.
Since $\{ s^\ast s \mid s \in S \}$ is linearly ordered, it follows from Theorem \ref{cyclic-phz} that $S$ is schurian.
\end{proof}

\section{Proof of theorem \ref{thm:2-2-s} }\label{sec:main2}

We denote $\Or(S) \Ot(S)$ and $\Ot(S) \cap \Or(S)$ by $T$ and $H$, respectively.
We put
\[ I := \{ 0, 1, \dots, p^2-1 \}, I^\circ := \{ 0, p, 2p, \dots, (p-1)p \}, I^\times := I \setminus I^\circ,\]
\[J := \{ 0, 1, \dots, p-1 \} ~~\text{and}~~ J^\times := J \setminus \{ 0 \}.\]

Since $S/\!/\Or(S) \cong C_{p^2} ~\text{or}~ C_p \times C_p$, without loss of generality, we can assume
\[S = \bigcup_{i, j \in J}\Or(S)s_it_j\]
such that $s_1 \in S \setminus T$, $t_1 \in \Ot(S) \setminus H$, $\Or(S)s_1^i = \Or(S)s_i$ and
$\Or(S)t_1^i = \Or(S)t_i$ for each $i \in J$.
Note that $Ts_1^i = \bigcup_{j \in J} \Or(S)s_it_j$ for each $i \in J$.

\vskip5pt
From now on, we fix $\delta \in X$ and $r_0 \in \Or(S) \setminus H$.
We denote $\delta \Or(S)s_1^it_1^j$ by $X_{i+jp}$.

\begin{rem}
We have $T = \bigcup_{j \in J} \Or(S)t_1^j = \bigcup_{j \in J} \Ot(S)r_0^j$.
\end{rem}

For each $i \in J$, we consider the set of equivalence classes on $\bigcup_{j \in I^\circ}  X_{i+j}$ induced by $\bigcup_{t \in \Ot(S)} t \cap (\bigcup_{j \in I^\circ} X_{i+j} \times \bigcup_{j \in I^\circ} X_{i+j})$.
Denote it by $\{E_{i,j} \mid j \in J \}$.

For each $i \in J$, we take
\[\beta_{i,j} \in X_i ~~(j \in J)\]
such that
\[\beta_{i,j} \in E_{i,j}  ~\text{and}~  \beta_{i,j+1} \in \beta_{i,j} r_0.\]
The subindex $j$ of $\beta_{i,j}$ is reduced by modulo $p$.
Note that $\bigcup_{l \in I^\circ} X_{i+l} = \{ \beta_{i,j} t \mid j \in J, t \in \Ot(S) \}$.

\begin{rem}
$\{ X_i \cap E_{i,j} \mid j \in J \}$ is the set of equivalence classes on $X_i$ induced by $\bigcup_{h \in H} h \cap (X_i \times X_i)$.
\end{rem}

\vskip5pt
Let $X_A := \bigcup_{i \in I} \{ X_{(i,j)} \mid j \in J \}$ be a partition of $X$,
where $\{ X_{(i,j)} \mid j \in J \}$ is the set of equivalence classes on $X_i$ induced by $\bigcup_{h \in H} h \cap (X_i \times X_i)$.
The subindex $j$ of $X_{(i,j)}$ is reduced by modulo $p$.

\begin{rem}
For $k, l \in J$, we have $\beta_{k,j}t_1^l \in X_{(k + lp,j)}$.
\end{rem}

\subsection{One-point stabilizer of the automorphism group}
\hfill \break
In this subsection, we will show that for any $s \in S$, $\mathrm{Aut}(S)_\delta$ is transitive on $\delta s$.

\vskip10pt
For a fixed $h_0 \in H \setminus \{ 1_X \}$,
define $\phi : X \rightarrow X$ such that
\begin{enumerate}
\item for each $i \in I^\circ$
\[ \phi|_{X_{(i,0)}} \text{~is the identity map,} \]
\item for each $(i,j) \in (I^\circ \times J^\times) \bigcup (I^\times \times J)$
\[ \phi|_{X_{(i,j)}} : X_{(i,j)} \rightarrow  X_{(i,j)} \text{~by~} \alpha^\phi = \alpha h_0. \]
\end{enumerate}

\begin{prop}\label{prop:one-point-1}
$\langle \phi \rangle$ is a nontrivial subgroup of $\mathrm{Aut}(S)_\delta$ such that
$\langle \phi \rangle$ is transitive on $X_{(i,j)}$ for each $(i, j) \in (I \times J) \setminus (I^\circ \times \{ 0\})$.
\end{prop}
\begin{proof}
It follows from the definition of $\phi$ that
$\langle \phi \rangle$ is transitive on $X_{(i,j)}$ for each $(i, j) \in (I \times J) \setminus (I^\circ \times \{ 0\})$.
It suffices to verify $X_{(i,j)} \sim_\phi X_{(i',j')}$ for all $(i, j), (i',j') \in I \times J$.

\vskip5pt
For $(i, j), (i',j') \in I^\circ \times \{0\}$,
we clearly have $X_{(i,j)} \sim_\phi X_{(i',j')}$ since $\phi|_{X_{(i,j)}}$ and $\phi|_{X_{(i',j')}}$ are the identity maps.

\vskip5pt
For $(i, j), (i',j') \in (I \times J) \setminus (I^\circ \times \{ 0\})$,
we have
\begin{eqnarray*}
r(x^\phi, y^\phi) & = &  r(xh_0, yh_0) = r(xh_0, x) r(x, y) r(y, yh_0) \\
                  & = &  h_0^\ast r(x,y) h_0 = r(x,y)
\end{eqnarray*}
for each $(x,y) \in X_{(i,j)} \times X_{(i,j)}$.

\vskip5pt
For $(i, j) \in I^\circ \times \{0\}$ and $(i',j') \in (I \times J) \setminus (I^\circ \times \{ 0\})$,
we divide our consideration into three cases.

(Case 1) $i=i'$ and $j \neq j'$ : we have $X_{(i,j)} \sim_\phi X_{(i',j')}$
since $X_{(i,j)} \times X_{(i',j')} \subseteq t$ for some $t \in \Or(S) \setminus H$.

\vskip5pt
(Case 2) $i \not\equiv i' (\mathrm{mod} ~p)$ :
Then either $j=j'$ or $j\neq j'$.
Whichever the case may be, we have $X_{(i,j)} \times X_{(i',j')} \subseteq s$ for some $s \in S \setminus T$.
So, $X_{(i,j)} \sim_\phi X_{(i',j')}$.

\vskip5pt
(Case 3) $i\neq i' $, $i \equiv i' (\mathrm{mod} ~p)$ and $j \neq j'$ :
we have $X_{(i,j)} \sim_\phi X_{(i',j')}$
since $X_{(i,j)} \times X_{(i',j')} \subseteq t$ for some $t \in T \setminus \Ot(S)$.


\end{proof}

Define $\psi : X \rightarrow X$ such that
\begin{enumerate}
\item for each $i \in I^\circ$ and $j \in J$
\[ \psi|_{X_{(i,j)}} = \phi|_{X_{(i,j)}}, \]
\item for each $i \in J^\times$ and $l, j \in J$
\[ \psi|_{X_{(i+lp,j)}} : X_{(i+lp,j)} \rightarrow  X_{(i+lp,j+1)} \text{~by~} (\beta_{i,j} t_1^{l}h)^\psi = \beta_{i,j+1} t_1^{l}h ~~(h \in H).\]
\end{enumerate}

\begin{prop}\label{prop:one-point-2}
$\langle \psi \rangle$ is a nontrivial subgroup of $\mathrm{Aut}(S)_\delta$ such that
for each $(l,i) \in I^\circ \times J^\times$, $\langle \psi \rangle$ is transitive on $\{ X_{(i+l,j)} \mid j \in J \}$.
\end{prop}
\begin{proof}
First of all, we prove that every pair of subsets in $\{ X_{(i,j)} \mid (i,j) \in I^\circ \times J \} \bigcup \{ X_i \mid i \in I^\times \}$
is compatible with respect to $\psi$.

\vskip5pt
Since $\psi|_{X_{(i,j)}} = \phi|_{X_{(i,j)}}$ for each $(i,j) \in I^\circ \times J$,
it follows from the proof of Proposition \ref{prop:one-point-1} that $X_{(i,j)} \sim_\psi X_{(i',j')}$ for $(i,j), (i',j') \in I^\circ \times J$.

\vskip5pt
For $(i,j) \in I^\circ \times J$ and $i' \in I^\times$, we have $X_{(i,j)} \sim_\psi X_{i'}$ since $X_{(i,j)} \times X_{i'} \subseteq s$
for some $s \in S \setminus T$.

\vskip5pt
For $i, i' \in I^\times$ $(i \not\equiv i' (\mathrm{mod} ~p))$, we have $X_i \sim_\psi X_{i'}$ since $X_i \times X_{i'} \subseteq s$
for some $s \in S \setminus T$.

\vskip5pt
For $i, i' \in I^\times$ $(i \equiv i' (\mathrm{mod} ~p))$,
let us put $i= i_1+ l$ and $i'= i_1+ l'$ $(i_1 \in J ~\text{and}~ l, l' \in I^\circ)$, we show $X_{i_1+l} \sim_\psi X_{i_1+l'}$.
Let $(x,y) \in X_{i_1 +l} \times X_{i_1 +l'}$.
Then $x= \beta_{i_1,j}t_1^lh_1$ and $y= \beta_{i_1,j'}t_1^{l'}h_2$ for some $j, j' \in J^\times$ and $h_1, h_2 \in H$.
We have
\begin{eqnarray*}
r(x, y)  & = & r(\beta_{i_1,j}t_1^lh_1, \beta_{i_1,j}) r(\beta_{i_1,j}, \beta_{i_1,j'}) r(\beta_{i_1,j'}, \beta_{i_1,j'}t_1^{l'}h_2)  \\
         & = & (t_1^lh_1)^\ast r(\beta_{i_1,j}, \beta_{i_1,j'}) t_1^{l'}h_2
\end{eqnarray*}
and
\begin{eqnarray*}
r(x^\psi, y^\psi) & = & r(\beta_{i_1,j+1}t_1^lh_1, \beta_{i_1,j+1}) r(\beta_{i_1,j+1}, \beta_{i_1,j'+1}) r(\beta_{i_1,j'+1}, \beta_{i_1,j'+1}t_1^{l'}h_2) \\
                  & = & (t_1^lh_1)^\ast r(\beta_{i_1,j+1}, \beta_{i_1,j'+1}) t_1^{l'}h_2.
\end{eqnarray*}

(Case 1) $j = j'$ : we have $r(x^\psi, y^\psi) = h_1^\ast (t_1^l)^\ast t_1^{l'} h_2 = r(x, y)$.

\vskip5pt
(Case 2) $j \neq j'$ : we have $r(x^\psi, y^\psi) = h_1^\ast (t_1^l)^\ast r_0^m t_1^{l'} h_2 = r(x, y)$ for some $m \in J^\times$.

\vskip5pt
Finally, it follows from the definition of $\psi$ that $\langle \psi \rangle$ is transitive on $\{ X_{(i+l,j)} \mid j \in J \}$.
\end{proof}

It follows from Propositions \ref{prop:one-point-1} and \ref{prop:one-point-2} that
for each $s \in S$, $\langle \phi, \psi \rangle$ is transitive on $\delta s$.

\subsection{Transitivity of the automorphism group}
\hfill \break
In this subsection, for all $\alpha, \beta \in X$ we will construct a faithful map $\phi : X \rightarrow X$
such that $\alpha^\phi = \beta$.
We divide our consideration into four cases :
\begin{enumerate}
\item[(I)] $\alpha H= \beta H$,
\item[(II)] $\alpha H \neq \beta H$ and $\alpha \Or(S)= \beta \Or(S)$,
\item[(III)] $\alpha \Or(S) \neq \beta \Or(S)$ and $\alpha T = \beta T$,
\item[(IV)] $\alpha T \neq \beta T$.
\end{enumerate}

In the case $\mathrm{(I)}$, without loss of generality, we assume $\alpha, \beta \in X_{(0,j)}$ for some $j$.

Put $\alpha_0=\alpha$ and $\beta_0=\beta$, and take $\alpha_i, \beta_i \in X$ such that
$\alpha_{i+1} = \alpha_i t_1$ and $\beta_{i+1} = \beta_i t_1$ for each $i \in J \setminus \{ p-1 \}$.

Define $\phi : X \rightarrow X$ such that
\begin{enumerate}
\item for each $l \in I^\circ$
\[\phi|_{X_{(l,j)}} : X_{(l,j)} \rightarrow  X_{(l,j)} \text{~by~} \alpha_{\frac{l}{p}} h \mapsto \beta_{\frac{l}{p}} h ~~(h \in H),\]
\item \[\phi|_{X \setminus \bigcup_{l \in I^\circ} X_{(l,j)}} \text{~is the identity map}.\]
\end{enumerate}

\begin{prop}\label{prop:tr1}
For $\alpha, \beta \in X$ $(\alpha H= \beta H)$, $\phi : X \rightarrow X$ is a faithful map
such that $\alpha^\phi = \beta$.
\end{prop}
\begin{proof}
Set $X_B = \{ X_{(l,j)} \mid l \in I^\circ \}$.
We divide $X_A$ into two parts, i.e., $X_B$ and $X_A \setminus X_B$.

\vskip5pt
For $X_{(i,j)},  X_{(i',j')} \in X_A \setminus X_B$, we clearly have $X_{(i,j)} \sim_\phi X_{(i',j')}$
since $\phi|_{X \setminus \bigcup_{l \in I^\circ} X_{(l,j)}}$ is the identity map.

\vskip5pt
For $X_{(l,j)} \in X_B$ and $X_{(i,k)} \in X_A \setminus X_B$,

(Case 1) $i \in I^\circ$ : we have $X_{(l,j)} \sim_\phi X_{(i,k)}$ since $X_{(l,j)} \times X_{(i,k)} \subseteq s$ for some $s \in T \setminus \Ot(S)$.

\vskip5pt
(Case 2) $i \in I^\times$ : we have $X_{(l,j)} \sim_\phi X_{(i,k)}$ since $X_{(l,j)} \times X_{(i,k)} \subseteq s$ for some $s \in S \setminus T$.

\vskip5pt
For $X_{(l,j)},  X_{(l',j)} \in X_B$, let us take $(x,y) \in X_{(l,j)} \times  X_{(l',j)}$.
Then $x = \alpha_{\frac{l}{p}}h_1$ and $y = \alpha_{\frac{l'}{p}}h_2$ for some $h_1, h_2 \in H$.
We have
\[r(x,y) = r(\alpha_{\frac{l}{p}}h_1, \alpha_{\frac{l}{p}}) r(\alpha_{\frac{l}{p}}, \alpha_{\frac{l'}{p}}) r(\alpha_{\frac{l'}{p}}, \alpha_{\frac{l'}{p}}h_2)\]
and
\[r(x^\phi, y^\phi) = r(\beta_{\frac{l}{p}}h_1, \beta_{\frac{l}{p}}) r(\beta_{\frac{l}{p}}, \beta_{\frac{l'}{p}}) r(\beta_{\frac{l'}{p}}, \beta_{\frac{l'}{p}}h_2).\]

(Case 1) $l = l'$ : we have $r(x^\phi, y^\phi) = h_1^\ast h_2 = r(x,y)$.

\vskip5pt
(Case 2) $l \neq l'$ : we have $r(x^\phi, y^\phi) = h_1^\ast t_1^m h_2 = r(x,y)$ for some $m \in J^\times$.
\end{proof}

\vskip10pt
In the case $\mathrm{(II)}$, without loss of generality, we assume $\alpha, \beta \in X_{0}$ such that $\alpha \in X_{(0,0)}$ and $\beta \in X_{(0,1)}$.
By Proposition \ref{prop:tr1}, we can assume $\alpha = \beta_{0,0}$ and $\beta = \beta_{0,1}$.

Define $\phi : X \rightarrow X$ such that
\begin{enumerate}
\item for each $i, j \in J$
\[ \phi|_{X_{(ip,j)}} : X_{(ip,j)} \rightarrow  X_{(ip,j+1)} \text{~by~} (\beta_{0,j} t_1^i h)^\phi = \beta_{0,j+1} t_1^i h ~~(h \in H),\]
\item \[\phi|_{\bigcup_{i \in I^\times} X_i} \text{~is the identity map}.\]
\end{enumerate}

\begin{prop}\label{prop:tr2}
For $\alpha, \beta \in X$ $(\alpha H \neq \beta H ~\text{and}~ \alpha \Or(S)= \beta \Or(S))$, $\phi : X \rightarrow X$ is a faithful map
such that $\alpha^\phi = \beta$.
\end{prop}
\begin{proof}

For $i, i' \in I^\times$, we clearly have $X_i\sim_\phi X_{i'}$
since $\phi|_{\bigcup_{i \in I^\times} X_i}$ is the identity map.

\vskip5pt
For $i \in I^\circ$ and $i' \in I^\times$, we have $X_i\sim_\phi X_{i'}$ since $X_i \times X_{i'} \subseteq s$ for some $s \in S \setminus T$.

\vskip5pt
For $ip, i'p \in I^\circ$, let us take $(x,y) \in X_{ip} \times X_{i'p}$.
Then $x = \beta_{0,j} t_1^ih_1$ and $y = \beta_{0,j'} t_1^{i'}h_2$ for some $j, j' \in J$ and $h_1, h_2 \in H$.
We have
\[r(x,y) = r(\beta_{0,j} t_1^ih_1, \beta_{0,j}) r(\beta_{0,j}, \beta_{0,j'}) r(\beta_{0,j'},\beta_{0,j'} t_1^{i'}h_2)\]
and
\[r(x^\phi, y^\phi) = r(\beta_{0,j+1} t_1^ih_1, \beta_{0,j+1}) r(\beta_{0,j+1}, \beta_{0,j'+1}) r(\beta_{0,j'+1},\beta_{0,j'+1} t_1^{i'}h_2).\]
Since $r(\beta_{0,j}, \beta_{0,j'}) = r(\beta_{0,j+1}, \beta_{0,j'+1})$, we have $r(x^\phi, y^\phi) = r(x,y)$.
\end{proof}

\vskip10pt
In the case $\mathrm{(III)}$, without loss of generality, we assume $\alpha \in X_0, \beta \in X_p$.
By Propositions \ref{prop:tr1} and \ref{prop:tr2}, we can assume $\alpha = \beta_{0,0}$ and $\beta = \beta_{0,0}t_1$.

Define $\phi : X \rightarrow X$ such that
for each $i, j \in J$
\[ \phi|_{\bigcup_{l \in I^\circ} X_{(i+l,j)}} : \bigcup_{l \in I^\circ}X_{(i+l,j)} \rightarrow  \bigcup_{l \in I^\circ}X_{(i+l,j)} \text{~by~} (\beta_{i,j} t)^\phi = \beta_{i,j}t_1t ~~(t \in \Ot(S)).\]

\begin{prop}\label{prop:tr3}
For $\alpha, \beta \in X$ $(\alpha \Or(S) \neq \beta \Or(S) ~\text{and}~ \alpha T= \beta T)$, $\phi : X \rightarrow X$ is a faithful map
such that $\alpha^\phi = \beta$.
\end{prop}
\begin{proof}

For $i, i' \in J$, let us take $(x,y) \in X_i \times X_{i'}$.
Then $x = \beta_{i,j} l_1$ and $y = \beta_{i',j'} l_2$ for some $j, j' \in J$ and $l_1, l_2 \in \Ot(S)$.
We have
\[r(x,y) = r(\beta_{i,j} l_1, \beta_{i,j}) r(\beta_{i,j}, \beta_{i',j'}) r(\beta_{i',j'}, \beta_{i',j'} l_2)\]
and
\[r(x^\phi, y^\phi) = r(\beta_{i,j} t_1l_1, \beta_{i,j}) r(\beta_{i,j}, \beta_{i',j'}) r(\beta_{i',j'}, \beta_{i',j'} t_1l_2).\]

(Case 1) $i \neq i'$ : we have $r(x^\phi, y^\phi) =  l_1^\ast t_1^\ast s_1^k t_1 l_2 = l_1^\ast s_1^k l_2 = r(x,y)$ for some $k \in J^\times$.

\vskip5pt
(Case 2) $i = i'$ and $j \neq j'$ : we have $r(x^\phi, y^\phi) = l_1^\ast t_1^\ast r_0^k t_1 l_2 = l_1^\ast r_0^k l_2 = r(x,y)$ for some $k \in J^\times$.

\vskip5pt
(Case 3) $i = i'$ and $j = j'$ : we have $r(x^\phi, y^\phi) = l_1^\ast l_2 = r(x,y)$.

\end{proof}

\vskip10pt
In the case $\mathrm{(IV)}$,
without loss of generality, we assume $\alpha \in X_0, \beta \in X_1$.
By Propositions \ref{prop:tr1}, \ref{prop:tr2} and \ref{prop:tr3},
we can assume $\alpha = \beta_{0,0} \in X_{(0,0)}$ and $\beta =  \beta_{1,0} \in X_{(1,0)}$.

Define $\phi : X \rightarrow X$ such that
for each $i \in J$
\[\phi|_{\bigcup_{l \in I^\circ} X_{i+l}} : \bigcup_{l \in I^\circ}X_{i+l} \rightarrow  \bigcup_{l \in I^\circ}X_{i+1+l} \text{~by~} (\beta_{i,j} t)^\phi = \beta_{i+1,j} t ~~(j \in J, t \in \Ot(S)).\]

\begin{prop}\label{prop:tr4}
For $\alpha, \beta \in X$ $(\alpha T \neq \beta T)$, $\phi : X \rightarrow X$ is a faithful map
such that $\alpha^\phi = \beta$.
\end{prop}
\begin{proof}

For $i, i' \in J$, let us take $(x,y) \in X_i \times X_{i'}$.
Then $x = \beta_{i,j} l_1$ and $y = \beta_{i',j'} l_2$ for some $j, j' \in J$ and $l_1, l_2 \in \Ot(S)$.
We have
\[r(x,y) = r(\beta_{i,j} l_1, \beta_{i,j}) r(\beta_{i,j}, \beta_{i',j'}) r(\beta_{i',j'}, \beta_{i',j'} l_2)\]
and
\[r(x^\phi, y^\phi) = r(\beta_{i+1,j} l_1, \beta_{i+1,j}) r(\beta_{i+1,j}, \beta_{i'+1,j'}) r(\beta_{i'+1,j'}, \beta_{i'+1,j'} l_2).\]

(Case 1) $i \neq i'$ : we have $r(x^\phi, y^\phi) = l_1^\ast s_1^k l_2 = r(x,y)$ for some $k \in J^\times$.

\vskip5pt
(Case 2) $i = i'$ and $j \neq j'$ : we have $r(x^\phi, y^\phi) = l_1^\ast r_0^k l_2 = r(x,y)$ for some $k \in J^\times$.

\vskip5pt
(Case 3) $i = i'$ and $j = j'$ : we have $r(x^\phi, y^\phi) = l_1^\ast l_2 = r(x,y)$.
\end{proof}

It follows from Propositions \ref{prop:tr1} -- \ref{prop:tr4} that $\mathrm{Aut}(S)$ is transitive on $X$.

\bibstyle{plain}

\end{document}